\documentclass[11 pt]{article}%
\usepackage{amsmath, amsfonts, amsthm, color,latexsym}
\usepackage{amsmath}
\usepackage{amsfonts}
\usepackage{amssymb}
\usepackage{graphicx}%
\providecommand{\U}[1]{\protect\rule{.1in}{.1in}}
\newtheorem{theorem}{Theorem}[section]
\newtheorem{proposition}[theorem]{Proposition}
\newtheorem{corollary}[theorem]{Corollary}

\newtheorem{remark}[theorem]{Remark}

\newtheorem{lemma}[theorem]{Lemma}
\newtheorem{final remark}[theorem]{Final Remark}
\newtheorem{definition}[theorem]{Definition}
\begin{document}

\title{Summability of multilinear mappings: Littlewood, Orlicz\linebreak and beyond}
\date{}
\author{Oscar Blasco\thanks{Supported by MEC and FEDER Projects MTM2005-08350-C03-03 and MTM2008-04594/MTM.}\,,
Geraldo Botelho, Daniel
Pellegrino\thanks{Supported by CNPq Project 471054/2006-2 and Grant 308084/2006-3.}%
~ and Pilar Rueda\thanks{Supported by MEC and FEDER Project
MTM2005-08210.\hfill\newline2000 Mathematics Subject Classification. Primary
46G25; Secondary 47B10. \newline Keywords: absolutely summing multilinear
mapping, type and cotype, Littlewood Theorem.}}
\maketitle

\begin{abstract}
In this paper we prove a plenty of new results concerning
summabililty properties of multilinear mappings between Banach
spaces, such as an extension of Littlewood's $4/3$ Theorem. Among
other features, it is shown that every continuous $n$-linear form on
the disc algebra $\mathcal{A}$ or the Hardy space
$\mathcal{H}^{\infty}$ is $(1;2,\ldots,2)$-summing, the role of the
Littlewood-Orlicz property in the theory is established and the
interplay with almost summing multilinear mappings is explored.
\end{abstract}

\section*{Introduction}

Motivated by several matters related to linear functional analysis,
such as integral equations, Fourier analysis and analytic number
theory, the theory of multilinear forms and polynomials on Banach
spaces was initiated in the beginning of the last century with the
works of several outstanding mathematicians like Banach, Bohr,
Bohnenblust, Hille, Littlewood,
Orlicz, Schur, etc. In 1930, Littlewood \cite{littlewood} proved a celebrated theorem asserting that%
\[
\left(
{\displaystyle\sum\limits_{i,j=1}^{\infty}}
\left\vert A(e_{i},e_{j})\right\vert ^{\frac{4}{3}}\right)  ^{\frac{3}{4}}%
\leq\sqrt{2}\left\Vert A\right\Vert
\]
for every continuous bilinear form $A$ on $c_0\times c_{0}.$ One year later, Bohnenblust and Hille
\cite{Bohnenblust}\ realized the importance of this result to the convergence of ordinary Dirichlet
series and extended Littlewood's result to multilinear mappings in the following fashion:

If $A$ is a continuous $n$-linear form on $c_0\times \cdots\times c_{0}$, then there is a constant
$C_{n}$ (depending only on $n$) such that%
\[
\left(
{\displaystyle\sum\limits_{i_{1},...,i_{n}=1}^{\infty}}
\left\vert A(e_{i_{1}},...,e_{i_{n}})\right\vert ^{\frac{2n}{n+1}}\right) ^{\frac{n+1}{2n}}\leq
C_{n}\left\Vert A\right\Vert .
\]

These results can be regarded as the beginning of the study of summability properties of multilinear mappings
between Banach spaces. This line of investigation has been developed since then and more recently it has
found its place within the theory of ideals of multilinear mappings outlined by Pietsch \cite{Pie} in 1983.
In this context classes of absolutely summing multilinear mappings are studied as generalizations of the very
successful theory of absolutely summing linear operators. The theory has been successfully developed by
several authors (a list of references is omitted because it would grow very large) and even applications to
Quantum Mechanics have been recently found (see \cite{PWPVJ}). One of the trends of the theory of absolutely
summing multilinear mappings is the search for summability properties in the spirit of those of Littlewood's
and Bohnenblust-Hille's theorems (see, e.g., \cite{alrt, Junek-preprint, Indag, choi, MT, Thesis, david}). In
this paper we aim to give new contributions to this line of investigation in several directions, which we
describe next.

Two well known results related to the linear theory are
Grothendieck's theorem, that asserts that every continuous linear
operator from $\ell_{1}$ to $\ell_{2}$ is absolutely summing, and
the weak Dvoretzky-Rogers theorem, that asserts that the identity
operator of any infinite dimensional Banach space fails to be
absolutely $p$-summing for any $1\leq p<\infty$. These two important
results can be considered as the roots of what has been known as
coincidence and non-coincidence results. The passage from the linear
to the multilinear case has occasioned the emergence of several
coincidence and non-coincidence situations for absolutely summing
multilinear mappings (see \cite{Alencar, bote97, pams, Port, Indag,
Pe1,Pe2,Pe3, Thesis, david}). The scope of the present paper is to
prove new coincidence theorems, some of them generalizing known
results and some giving new perspectives to the subject.\\
\indent Respecting the historical development of the subject we start in Section \ref{43} by extending the
classical Littlewood $4/3$-Theorem by proving that, given $1\leq p\leq2$ and
$\frac{1}{q}=\frac{1}{2}+\frac{1}{p^{\prime}}$, any continuous bilinear functional $A$ defined on
$c_{0}\times c_{0}$ satisfies that $(A(e_{j},e_{k}))_{jk}$ belongs to $\ell_{p}(\ell_{q})$, where
$(e_{j})_{j}$ is the unit basis. Actually we prove a more general version of this result, in which by taking
$p=4/3$ we recover Littlewood's theorem. In Section \ref{coincidence} we prove coincidence and inclusion
theorems that will be useful in later sections. While the role of the Orlicz property in the theory is well
established, in Section \ref{lop} we show that the Littlewood-Orlicz property can be used to get even
stronger
results. More precisely, we prove that for suitable $n,p_{1},p_{2}%
,\ldots,p_{n}$, any continuous $n$-linear mapping defined on a product of Banach spaces, one of which has a
dual with the Littlewood-Orlicz property, is absolutely $(p;p_{1},\ldots,p_{n})$-summing. We also generalize
a coincidence result due to P\'{e}rez-Garc\'{\i}a (see Theorem \ref{pggen}). Inspired by this generalization,
in Section \ref{1r} we develop a general technique of extending bilinear coincidences to $n$-linear
coincidences, $n\geq3$. In Section \ref{almost} almost summing operators are used to get some more
summability properties of multilinear mappings. Calling on the type/cotype theory we get for instance that
for any $1\leq p \leq 2$, every continuous bilinear functional $A$ defined on $\ell_{p}\times F$, where $F$
is a Banach space whose dual has type 2, is absolutely $(p;2,1)$-summing. Moreover, if $1\leq p\leq2$ then
$A$ is absolutely $(r_{p};r_{p},r_{p})$-summing for any $1\leq r_{p}\leq \frac{2p}{3p-2}$.

\section{Notation and background}

Henceforth $E_{1},\ldots,E_{n},E,F$ will be Banach spaces over the scalar
field $\mathbb{K}=\mathbb{R}$ or $\mathbb{C}$, $B_{E}$ represents the closed
unit ball of $E$ and the topological dual of $E$ will be denoted by
$E^{\prime}$. The Banach space of all continuous $n$-linear mappings from
$E_{1}\times\cdots\times E_{n}$ into $F$ is denoted by $\mathcal{L}%
(E_{1},\ldots,E_{n};F)$. As usual we write $\mathcal{L}(^{n}E;F)$ if
$E_{1}=\cdots=E_{n}=E$. For the general theory of
polynomials/multilinear mappings between Banach spaces we refer to
\cite{Dineen, Mujica}.

Let $p > 0$. By $\ell_{p}(E)$ we denote the Banach space of all absolutely $p$-summable sequences
$(x_{j})_{j=1}^{\infty}$ in $E$ endowed with its usual $\ell_p$-norm ($p$-norm if $0 < p < 1$).
Let $\ell_{p}^{w}(E)$ be the space of those sequences
$(x_{j})_{j=1}^{\infty}$ in $E$ such that $(\varphi(x_{j}))_{j=1}^{\infty}%
\in\ell_{p}$ for every $\varphi\in E^{\prime}$ endowed with the norm ($p$-norm if $0 < p < 1$)
\[
\|(x_{j})_{j=1}^{\infty}\|_{\ell_{p}^{w}(E)}=\sup_{\varphi\in B_{E^{\prime}}%
}(\sum_{j=1}^{\infty}|\varphi(x_{j})|^{p})^{\frac1p}.
\]
Let $\ell_{p}^{u}(E)$ denote the closed subspace of $\ell_{p}^{w}(E)$ formed
by the sequences $(x_{j})_{j=1}^{\infty}\in\ell_{p}^{w}(E)$ such that %
$\lim_{k\rightarrow\infty}\Vert(x_{j})_{j=k}^{\infty}\Vert_{\ell_{p}^{w}%
(E)}=0.$

Let $A\in{\mathcal{L}}(E_{1},\ldots,E_{n};F)$ and let $X_{1}\ldots,X_{n},Y$ be
spaces of sequences in $E_{1},\ldots,E_{n},F$ respectively. Whenever we say
that $\hat{A}\colon X_{1}\times\cdots\times X_{n}\longrightarrow Y$ is bounded
we mean that the correspondence
\[
((x_{j}^{1})_{j=1}^{\infty},\ldots,(x_{j}^{n})_{j=1}^{\infty})\in X_{1}%
\times\cdots\times X_{n}\mapsto
\]%
\[
\hat{A}((x_{j}^{1})_{j=1}^{\infty},\ldots,(x_{j}^{n})_{j=1}^{\infty
}):=(A(x_{j}^{1},\ldots,x_{j}^{n}))_{j=1}^{\infty} \in Y
\]
is well defined into $Y$ (hence multilinear) and continuous.

For $0< p,p_{1},p_{2},\ldots,p_{n}\leq\infty$ , we assume that $\frac{1}%
{p}\leq\frac{1}{p_{1}}+\cdots+\frac{1}{p_{n}}$. A multilinear
mapping $A\in{\mathcal{L}}(E_{1},\ldots,E_{n};F)$ is
\textit{absolutely $(p;p_{1},p_{2},\ldots,p_{n})$-summing} if there
exists $C>0$ such that
\[
\Vert(A(x_{j}^{1},x_{j}^{2},\ldots,x_{j}^{n}))_{j}\Vert_{p}\leq C\prod
_{i=1}^{n}\Vert(x_{j}^{i})_{j}\Vert_{\ell_{p_{i}}^{w}(E_{i})}%
\]
for all finite family of vectors $x_{j}^{i}$ in $E_{i}$ for $i=1,2,\ldots,n$.
The infimum of such $C>0$ is called the $(p;p_{1},\ldots,p_{n} )$-summing norm
of $A$ and is denoted by $\pi_{(p;p_{1},\ldots,p_{n})}(A)$. Let $\Pi_{(p;p_{1}
,p_{2},\ldots,p_{n})}(E_{1},\ldots,E_{n};F)$ denote the space of all
absolutely $(p;p_{1},p_{2},\ldots,p_{n})$-summing $n$-linear mappings from
$E_{1}\times\cdots\times E_{n}$ to $F$ endowed with the norm $\pi
_{(p;p_{1}\ldots,p_{n})}$. Thus, $A\in\Pi_{(p;p_{1},p_{2},\ldots,p_{n})}%
(E_{1},\ldots,E_{n};F)$ if and only if
\[
\hat{A}\colon\ell_{p_{1}}^{w}(E_{1})\times\ell_{p_{2}}^{w}(E_{2})\times
\cdots\times\ell_{p_{n}}^{w}(E_{n})\rightarrow\ell_{p}(F) \mbox{ is bounded.}
\]

It is well known that we can replace $\ell_{p_{k}}^{w}(E_{k})$ by $\ell
_{p_{k}}^{u}(E_{k})$ in the definition of absolutely summing mappings.

Absolutely $(\frac{p}{n};p, \ldots, p)$-summing $n$-linear mappings are
usually called \textit{$p$-dominated}. They satisfy the following
factorization result (see \cite[Theorem 13]{Pie}):

$A \in\Pi_{(\frac{p}{n};p, \ldots, p)}(E_{1},\ldots,E_{n};F)$ if and only if
there are Banach spaces $G_{1}, \ldots, G_{n}$, operators $u_{j} \in\Pi
_{p}(E_{j};G_{j})$ and $B \in\mathcal{L}(G_{1},\ldots,G_{n};F)$ such that
\begin{equation}
\label{fact}A = B \circ(u_{1}, \ldots, u_{n}).
\end{equation}

Let us now recall some basic facts about Rademacher functions and its use in
Banach space theory. For each $1\leq p\leq\infty$, we denote by $Rad_{p}(E)$
the space of sequences $(x_{j})_{j=1}^{\infty}$ in $E$ such that
\[
\Vert(x_{j})_{j=1}^{\infty}\Vert_{Rad_{p}(E)}=\sup_{n\in\mathbb{N}}%
\|\sum_{j=1}^{n}r_{j}x_{j}\|_{L^{p}([0,1],E)}<\infty,
\]
where $(r_{j})_{j\in\mathbb{N}}$ are the Rademacher functions on $[0,1]$
defined by $r_{j}(t)=sign(\sin2^{j}\pi t)$. The reader is referred to
\cite{Diestel, TJ, VTC} for the difference between this space and the space of
sequences $(x_{n})$ for which the series $\sum_{n=1}^{\infty}x_{n}r_{n}$ is
convergent in $L^{p}([0,1],E)$. It is easy to see that $Rad_{\infty}(E)$
coincides with $\ell_{1}^{w}(E)$. Making use of the Kahane's inequalities (see
\cite[p. 211]{Diestel}) it follows that the spaces $Rad_{p}(E)$ coincide up to
equivalent norms for all $1\leq p<\infty$. The unique vector space so obtained
will therefore be denoted by $Rad(E)$, and we agree to (mostly) use the norm
$\Vert\cdot\Vert_{Rad(E)}:=\Vert\cdot\Vert_{Rad_{2}(E)}$ on $Rad(E).$

Recall also that a linear operator $u\colon E\rightarrow F$ is said to be
\textit{almost summing} if there is a $C>0$ such that we have%
\[
\Vert\left(  u(x_{j})\right)  _{j=1}^{m}\Vert_{Rad(F)} \leq C\left\Vert
(x_{j})_{j=1} ^{m}\right\Vert _{\ell^{w}_{2}(E)}%
\]
for any finite set of vectors $\{x_{1},\ldots,x_{m}\}$ in $E$. The space of
all almost summing linear operators from $E$ to $F$ is denoted by $\Pi
_{a.s}(E;F)$ and the infimum of all $C>0$ fulfilling the above inequality is
denoted by $\|u\|_{a.s}$. Note that this definition differs from the
definition of almost summing operators given in \cite[p. 234]{Diestel} but
coincides with the characterization which appears a few lines after that
definition (yes, the definition and the stated characterization are not
equivalent). Since the proof of \cite[Proposition 12.5]{Diestel} uses the
characterization (which is our definition) we can conclude that every
absolutely $p$-summing linear operator, $1 \leq p < + \infty$, is almost summing.

The concept of almost summing multilinear mappings was considered in
\cite{Nach, Archiv} as reads as follows: A multilinear map $A\in{\mathcal{L}%
}(E_{1},...,E_{n};F)$ is said to be \textit{almost summing} if there
exists $C>0$ such that we have
\begin{equation}
\label{almostsumming}\| \left(  A(x^{1}_{j},...,x^{n}_{j})\right)  _{j=1}%
^{m}\|_{Rad(F)}\le C\prod_{i=1}^{n} \Vert(x_{j}^{i})_{j=1}^{m}\Vert
_{\ell_{2}^{w}(E_{i})}%
\end{equation}
for any finite set of vectors $(x^{i}_{j})_{j=1}^{m} \subset E_{i}$ for
$i=1,...,n$. We write $\Pi_{a.s}(E_{1},...,E_{n};F)$ for the space of almost
summing multilinear maps, which is
endowed with the norm
\[
\|A\|_{as}:= \inf\{C > 0\  \mbox{such\ that\ (\ref{almostsumming})\ holds} \}.
\]

For the theory of type and cotype in Banach spaces the reader is referred to \cite[Chapter 11]{Diestel}.
Recall that a Banach space $E$ is said to have the \textsl{Orlicz
property } if there exists a constant $C>0$ such that
\[
\left(  \sum_{j=1}^{n}\Vert x_{j}\Vert^{2}\right)  ^{1/2}\leq C\sup
_{t\in\lbrack0,1]}||\sum_{j=1}^{n}x_{j}r_{j}(t)||
\]
for any finite family $x_{1},x_{2},\dots x_{n}$ of vectors in $E.$
In other words, $E$ has the Orlicz property when the identity
operator $id_{E}$ is absolutely $(2;1)$-summing.

One should notice that, due to results by Talagrand (see \cite{T1,T2}), while
the Orlicz property is weaker than cotype $2$, having cotype $q>2$ is
equivalent to the existence of a constant $C>0$ such that
\[
\left(  \sum_{j=1}^{n}\Vert x_{j}\Vert^{q}\right)  ^{1/q}\leq C\sup
_{t\in\lbrack0,1]}||\sum_{j=1}^{n}x_{j}r_{j}(t)||
\]
for any finite family $x_{1},x_{2},\dots x_{n}$ of vectors in $E.$

A relevant property for our purposes is the following: We say that a Banach
space $E$ has the \textsl{Littlewood-Orlicz property} if $\ell_{1}^{w}(E)$ is
continuously contained in the projective tensor product $\ell_{2}\otimes_{\pi
}E$ (for a related concept of Littlewood-Orlicz operator we refer to
\cite[Section 4]{Bu}). Of course, since $\ell_{2}\otimes_{\pi}E\subset\ell
_{2}(E)$, the Littlewood-Orlicz property implies Orlicz-property.

In \cite{Cohen}, J. S. Cohen introduces the space
\[
\ell_{p}\left\langle E\right\rangle :=\left\{  (x_{n})_{n=1}^{\infty}\subset
E:\sum_{n=1}^{\infty}\left\vert x_{n}^{\ast}(x_{n})\right\vert <\infty\text{
for each }(x_{n}^{\ast})_{n=1}^{\infty}\in\ell_{p^{\prime}}^{w}(E^{\prime
})\right\}  ,
\]
where $1/p+1/p^{\prime}=1$ and the space of operators $p$-Cohen-nuclear
$u\in{\mathcal{L}}(E,F)$ such that
\begin{equation}
\label{cohen}\Vert(u(x_{j}))_{j=1}^{m}\Vert_{l_{p}\left\langle F\right\rangle
}\leq C\Vert(x_{j})_{j=1}^{m}\Vert_{\ell_{p}^{w}(E)}%
\end{equation}
for all finite family of vectors $x_{1},x_{2},\ldots,x_{m}$ in $E$.

It was first shown that
$\ell_{p}\otimes_{\pi}E\subset\ell_{p}\left\langle E\right\rangle $
(see \cite[Theorem 1.1.3 (i)]{Cohen}) and, later the space
$\ell_{p}\left\langle E\right\rangle $ was shown to coincide with
$\ell _{p}\otimes_{\pi}E$ (see \cite[Theorem 1]{Bu-D} or \cite{AF})
for $1<p<\infty$.

The reader is referred to \cite{AB1} for a description in terms of
integral operators, where $\ell_{p}\left\langle E\right\rangle $ is
denoted $\ell _{\pi_{1,p^{\prime}}}(E)$, and for a proof of
$\ell_{2}\left\langle E\right\rangle \subset Rad(E)$. Therefore we
always have
\[
(\ell_{2}\otimes_{\pi}E))\cap\ell_{1}^{w}(E) \subset Rad(E) \subset\ell
_{2}^{w}(E).
\]

The following result was obtained in \cite[Theorem 9]{AB1}:
\begin{equation}
\label{GT}\ell_{2}\otimes_{\pi}E=Rad(E)\Longleftrightarrow E \hbox{ is
a GT-space  of cotype }2
\end{equation}
where $E$ being a $GT$-space means that every continuous linear
mapping from $E$ to $\ell_{2}$ is absolutely $1$-summing. In
particular every GT-space with cotype $2$ has the Littlewood-Orlicz
property. The basic examples are ${\mathcal{L}}_{1}$-spaces and
other examples of GT-spaces with cotype $2$ can be found in
\cite{Pi}.

Let us end this preliminary section by mentioning that the complex
interpolation method, for which the reader is referred to
\cite[Theorem 5.1.2]{Bergh} or \cite[Theorem 3.1]{TJ}, and a
complexification technique (see \cite[Section IV.2]{Thesis}) will be
applied several times in Section \ref{1r}. The complexification
technique will allow us reduce proofs to the complex case. Similar
applications of this interpolation-complexification argument can be
found in \cite{Junek-preprint, Junek, Thesis}.

\section{An extension of Littlewood's $4/3$ theorem}

\label{43}

Littlewood \cite{littlewood} proved that if $A\colon c_{0}\times
c_{0}\rightarrow\mathbb{K}$ is a continuous bilinear
form, then%
\begin{equation}
\Big( {\sum\limits_{j,k}} \big\vert A(e_{j},e_{k})\big\vert ^{4/3}%
\Big) ^{3/4}\leq c\big\Vert A\big\Vert \label{ltw43}%
\end{equation}
with $c=\sqrt{2}.$ It is well-known that the constant $c=\sqrt{2}$
is far from being optimal, for example in \cite[Theorem 34.11]{DF}
or \cite[Theorem 11.11]{TJ} it is proved that in the complex case
the best constant $c$ satisfying (\ref{ltw43}) is dominated by $2^{1/4}%
K_{G}^{1/2}$, i.e.,
\begin{equation}
c\leq2^{1/4}K_{G}^{1/2}, \label{estL}%
\end{equation}
where $K_{G}$ is Grothendieck's constant (note that $2^{1/4}K_{G}^{1/2}%
<\sqrt{2}$ in the complex case since $K_{G}<\sqrt{2}$ in this case).
To the best of our knowledge the best estimate known for this
constant is $c \leq K_G$ \cite[Corollary 2, p. 280]{LP}.

In this section we extend Littlewood's Theorem in the complex case
to a more general setting in which the estimate for the best
constant remains $K_G$, that is, we improve the result keeping the
best known constant.

Given a matrix $m_{jk}$ we write
\[
\Vert(m_{jk})\Vert_{\ell_{p}(\ell_{q})}=\Big(\sum_{k}\big(\sum_{j}|m_{jk}%
|^{q}\big)^{p/q}\Big)^{1/p}.
\]
If $a$ and $\beta$ are matrices, we denote by $(\beta\circ a)_{jk}$ the
product of $\beta$ and $a$, that is
\[
(\beta\circ a)_{jk}=\sum_{l}\beta_{jl}a_{lk}.
\]

\begin{theorem}
\label{Littlewood} Let $A\in\mathcal{L}(^{2}c_{0};{\mathbb{C}})$. If
$a=(a_{jk})_{j,k}:=\left(  A(e_{j},e_{k})\right)  _{j,k}$, $1\leq p\leq2$ and
$\frac{1}{q}=\frac{1}{2}+\frac{1}{p^{\prime}}$, then
\[
\Vert(\beta\circ a)_{jk}\Vert_{\ell_{p}(\ell_{q})}\leq K_{G}\Vert A\Vert
\Vert(\beta_{jk})\Vert_{\ell_{\infty}(\ell_{2})},
\]
that is
\[
\Big(\sum_{k}\Big(\sum_{j}\big|\sum_{l}\beta_{jl}A(e_{l},e_{k})\big|^{q}
\Big)^{p/q}\Big)^{1/p}\leq K_{G}\Vert A\Vert\sup_{k}\Big(\sum_{j}|\beta
_{jk}|^{2}\Big)^{1/2}.
\]
In particular, selecting $\beta$ as the identity matrix,
\[
\Big(\sum_{k} \Big(\sum_{j}|A(e_{j},e_{k})|^{q}\Big)^{p/q}\Big)^{1/p} \leq
K_{G} \|A\|.
\]
Selecting $p=4/3$ we recover Littlewood's Theorem, that is $(A(e_{j}%
,e_{k}))_{jk}\in\ell_{4/3}(\mathbb{N}^{2}).$
\end{theorem}

\begin{proof} From \cite[Corollary 2.5]{david} we know that
\[
\sum_{j}|A(y_{j},x_{j})|\leq K_{G}\Vert A\Vert\Vert(x_{j})\Vert_{\ell_{2}%
^{w}(c_{0})}\Vert(y_{j})\Vert_{\ell_{2}^{w}(c_{0})}.
\]
Now write $x_{j}(k)=\lambda_{jk}$ and $y_{j}(k)=\beta_{jk}.$ Since the canonical basis of $\ell_{1}$ is a
norming set of $c_{0}$, from \cite[p. 36]{Diestel} we know that
\[
\Vert(x_{j})\Vert_{\ell_{2}^{w}(c_{0})}=\sup_{k}\Big(\sum_{j}|\lambda_{jk}%
|^{2}\Big)^{1/2}\text{ and }\Vert(y_{j})\Vert_{\ell_{2}^{w}(c_{0})}=\sup_{k}%
\Big(\sum_{j}|\beta_{jk}|^{2}\Big)^{1/2}.%
\]
So, we have
\[
\Big|\sum_{j}\sum_{k,l}\lambda_{jk}\varepsilon_{j}\beta_{jl}A(e_{l},e_{k})\Big|\leq K_{G}\Vert
A\Vert\sup_{k}\Big(\sum_{j}|\lambda_{jk}|^{2}\Big)^{1/2}\sup_{l}\Big(\sum
_{j}|\beta_{jl}|^{2}\Big)^{1/2}%
\]
for a convenient choice of $\varepsilon_{j}\in \mathbb{C}$ with $\left\vert
\varepsilon_{j}\right\vert =1.$
Note that
\[
(\lambda_{jk})=\left(  (\lambda_{jk})_{j=1}^{\infty}\right)  _{k=1}^{\infty
}\in\ell_{\infty}(\ell_{2})\text{ and }(\beta_{jl})=\left(  (\beta_{jl}%
)_{j=1}^{\infty}\right)  _{l=1}^{\infty}\in\ell_{\infty}(\ell_{2}).\text{ }%
\]
Hence
\[
\Big|\sum_{j}\sum_{k,l}\lambda_{jk}\varepsilon_{j}\beta_{jl}A(e_{l},e_{k})\Big|\leq K_{G}\Vert
A\Vert\Vert(\lambda_{jk})\Vert_{\ell_{\infty}(\ell_{2})}\Vert
(\beta_{jl})\Vert_{\ell_{\infty}(\ell_{2})}.
\]
Using the duality $(\ell_{1}(\ell_{2}))^{\ast}=\ell_{\infty}(\ell_{2})$ and the inequality
\[
\Big|\sum_{j,k}\big(\sum_{l}\beta_{jl}\varepsilon_{j}A(e_{l},e_{k})\big)\lambda_{jk}\Big|\leq
K_{G}\Vert A\Vert\Vert(\lambda_{jk})\Vert_{\ell_{\infty}(\ell_{2})}\Vert
(\beta_{jl})\Vert_{\ell_{\infty}(\ell_{2})}%
\]
one obtains that
\[
\Big(\sum_{l}\beta_{jl}\varepsilon_{j}A(e_{l},e_{k})\Big)_{jk}=\Big(  \big(\sum_{l}%
\beta_{jl}\varepsilon_{j}A(e_{l},e_{k})\big)_{j=1}^{\infty}\Big)  _{k=1}%
^{\infty}\in\ell_{1}(\ell_{2}),
\]
and also
\[
\sum_{k}\Big(\sum_{j}\Big|\big(\sum_{l}\beta_{jl}\varepsilon_{j}A(e_{l},e_{k}%
)\big)\Big|^{2}\Big)^{1/2}\leq K_{G}\Vert A\Vert\Vert(\beta_{jl})\Vert_{\ell_{\infty}%
(\ell_{2})}.
\]
Hence
\begin{equation}
\sum_{k}\Big(  \sum_{j}\Big\vert \sum_{l}\beta_{jl}A(e_{l},e_{k})\Big\vert ^{2}\Big) ^{1/2}\leq
K_{G}\Vert A\Vert\Vert(\beta_{jl})\Vert
_{\ell_{\infty}(\ell_{2})}.\label{lp}%
\end{equation}
But from \cite[Corollary 5.4.2 with $p=1$ and $q=2$]{Garling} we know that%
\begin{equation}
\Big(  \sum_{k}\big(  \sum_{j}\big\vert \sum_{l}\beta_{jl}A(e_{l}%
,e_{k})\big\vert \big)  ^{2}\Big)  ^{1/2}\leq\sum_{k}\Big( \sum_{j}\big\vert
\sum_{l}\beta_{jl}A(e_{l},e_{k})\big\vert ^{2}\Big)
^{1/2}.\label{lp2}%
\end{equation}
It follows that
\begin{align*}
\Vert(\beta\circ a)_{jk}\Vert_{\ell_{1}(\ell_{2})} &  \leq K_{G}\Vert
A\Vert\Vert(\beta_{jk})\Vert_{\ell_{\infty}(\ell_{2})}~{\rm and}\\
\Vert(\beta\circ a)_{jk}\Vert_{\ell_{2}(\ell_{1})} &  \leq K_{G}\Vert
A\Vert\Vert(\beta_{jk})\Vert_{\ell_{\infty}(\ell_{2}).}%
\end{align*}
Now complex interpolation gives that
\[
(\ell_{1}(\ell_{2}),\ell_{2}(\ell_{1}))_{[\theta]}=\ell_{p}((\ell_{2},\ell
_{1})_{[\theta]})
\]
with%
\[
\frac{1}{p}=(1-\theta)+\frac{\theta}{2}.
\]
One concludes that
\[
(\ell_{1}(\ell_{2}),\ell_{2}(\ell_{1}))_{[\theta]}=\ell_{p}(\ell_{q})
\]
with
\[
\frac{1}{p}=(1-\theta)+\frac{\theta}{2}\text{ and }\frac{1}{q}=\frac{1-\theta
}{2}+\theta\text{ (hence }\frac{1}{q}=\frac{1}{2}+\frac{1}{p^{\prime}%
}\text{).}%
\]
So $(\beta\circ a)_{jk}\in\ell_{p}(\ell_{q})$ and
\[
\Vert(\beta\circ a)_{jk}\Vert_{\ell_{p}(\ell_{q})}\leq K_{G}\Vert A\Vert
\Vert(\beta_{jk})\Vert_{\ell_{\infty}(\ell_{2})},
\]
that is
\[
\Big(\sum_{k}\big(\sum_{j}\big|\sum_{l}\beta_{jl} A(e_{l},e_{k})\big|^{q}\big)^{p/q}\Big)^{1/p}\leq
K_{G}\Vert A\Vert\sup_{k}\Big(\sum_{j}|\beta_{jk}|^{2}\Big)^{1/2}.
\]
Finally note that $\theta=1/2$ gives $p=4/3$ and $q=4/3$. \end{proof}

\begin{remark}\rm
\textrm{As pointed out before, although our result holds in a more
general setting, the estimate $K_G$ for the best Littlewood constant
we have just obtained in the complex case improves the estimate
$2^{1/4}K_{G}^{1/2} $given in \cite[Theorem 34.11]{DF} and
\cite[Theorem 11.11]{TJ} and equals the best known estimate. }
\end{remark}

\begin{remark}\rm
\textrm{Making $p=1$ and $q=2$ we recover the so-called {\it general
Littlewood inequality} that appears in \cite[(2.10), p. 280]{LP}. }
\end{remark}

\section{Some general coincidence results}

\label{coincidence}

Defant-Voigt Theorem (see \cite[Theorem 3.10]{Alencar}) stating that
\begin{equation}
\label{DV}{\mathcal{L}}(E_{1},\ldots, E_{n};\mathbb{K})=\Pi_{(1;1,\ldots,1)}(E_{1}%
,\ldots, E_{n};\mathbb{K})
\end{equation}
is probably the first and most folkloric coincidence result in the theory of
absolutely summing multilinear mappings. The next result gives a slightly more
general version.

\begin{proposition}
\label{primerlema} Let $n\geq2$ and $A\in{\mathcal{L}}(E_{1},\ldots
,E_{n};\mathbb{K})$. Then
\[
\hat{A}\colon Rad(E_{1})\times\cdots\times Rad(E_{n})\rightarrow\ell_{1}%
\]
is bounded. Moreover $\Vert\hat{A}\Vert=\Vert A\Vert.$
\end{proposition}

\begin{proof} Let $(x_{j}^{i})$ be finite sequences in $E_{i}$ for $i=1,\ldots,n.$
One can find a sequence $(\alpha_{j})$ of norm one scalars so that%
\[
\sum_{j}\left\vert A(x_{j}^{1},\ldots,x_{j}^{n})\right\vert =\sum_{j}%
A(\alpha_{j}x_{j}^{1},\ldots,x_{j}^{n}).
\]
Let
\begin{align*}
f_{\alpha}(t_{1})  &  =\sum_{j}\alpha_{j}r_{j}(t_{1})x_{j}^{i},\\
f_{i}(t_{i})  &  =\sum_{j}r_{j}(t_{i})x_{j}^{i},i=2,\ldots,n-1,\text{ and}\\
f_{n}(t_{1},\ldots,t_{n-1})  &  =\sum_{j}r_{j}(t_{1})\cdots r_{j}%
(t_{n-1})x_{j}^{n}%
\end{align*}
for $t_{1},\ldots,t_{n-1}\in\lbrack0,1]$. Using the orthogonality of the
Rademacher system and the Contraction Principle (see \cite[page 231]%
{Diestel}) we have%
\begin{align*}
&  \sum_{j}A(\alpha_{j}x_{j}^{1},\ldots,x_{j}^{n})\\
&  =\int_{0}^{1}\cdots\int_{0}^{1}A(f_{\alpha}(t_{1}),\ldots,f_{n-1}%
(t_{n-1}),f_{n}(t_{1},\ldots,t_{n-1}))dt_{1}\cdots dt_{n-1}\\
&  \leq\Vert A\Vert\int_{0}^{1}\cdots\int_{0}^{1}\cdots\left(  \int_{0}%
^{1}\Vert f_{\alpha}(t_{1})\Vert\Vert f_{n}(t_{1},\ldots,t_{n-1})\Vert
dt_{1}\right)  \Vert f_{2}(t_{2})\Vert\cdots\Vert f_{n-1}(t_{n-1})\Vert
dt_{2}\cdots dt_{n-1}\\
&  \leq\Vert A\Vert\int_{0}^{1}\cdots\int_{0}^{1}\Vert(x_{j}^{1})_{j}%
\Vert_{Rad_{2}}\Vert(x_{j}^{n})_{j}\Vert_{Rad_{2}}\Vert f_{2}(t_{2}%
)\Vert\cdots\Vert f_{n-1}(t_{n-1})\Vert dt_{2}\cdots dt_{n-1}\\
&  =\Vert A\Vert\Vert(x_{j}^{1})_{j}\Vert_{Rad_2}\Vert(x_{j}^{2}%
)_{j}\Vert_{Rad_1}\ldots\Vert(x_{j}^{n-1})_{j}\Vert_{Rad_1}\Vert(x_{j}^{n})_{j}\Vert_{Rad_2}\\
&  \leq\Vert A\Vert\Vert(x_{j}^{1})_{j}\Vert_{Rad}\cdots\Vert(x_{j}^{n}%
)_{j}\Vert_{Rad}.
\end{align*}
It is easy to see that $\Vert\hat{A}\Vert\geq\Vert A\Vert$ and we conclude the proof. \end{proof}

The following result, which appears in \cite[Proposition 3.3]{Thesis}, will be
used several times in this paper (we include a short proof for the sake of completeness):

\begin{proposition}
[Inclusion Theorem]Let $0< q\leq p\leq\infty$, $0< q_{j}\leq
p_{j}\leq\infty$ for all $j=1,\ldots,n$. If
$\frac{1}{q_{1}}+\cdots+\frac
{1}{q_{n}}-\frac{1}{q}\leq\frac{1}{p_{1}}+
\cdots+\frac{1}{p_{n}}-\frac{1}{p}$ then
\[
\Pi_{(q;q_{1},\ldots,q_{n})}(E_{1},\ldots,E_{n};F)\subset\Pi_{(p;p_{1}%
,\ldots,p_{n})}(E_{1},\ldots,E_{n};F)
\]
and $\pi_{(p;p_{1},\ldots,p_{n})}\leq\pi_{(q;q_{1} ,\ldots,q_{n})}$. \medskip
\end{proposition}

\begin{proof} By the monotonicity of the $\ell_p$-norms we may assume $\frac{1}{q_{1}}+\cdots+\frac{1}{q_{n}}-
\frac{1}{q}=\frac{1}{p_{1}}+ \cdots+\frac{1}{p_{n}}-\frac{1}{p}$. Let $A \in
\Pi_{(q;q_{1},\ldots,q_{n})}(E_{1},\ldots,E_{n};F)$ and $(x_j^k)_{j=1}^\infty \in \ell_{p_k}^w(E_k)$, $k = 1,
\ldots, n$, be given. We should prove that $(A(x_j^1, \ldots, x_j^n))_{j=1}^\infty \in \ell_p(F)$, and for
that it suffices to show that $(\alpha_j \cdot A(x_j^1, \ldots, x_j^n))_{j=1}^\infty \in \ell_q(F)$ for every
$(\alpha_j)_{j=1}^\infty \in \ell_r$ where $\frac{1}{p} + \frac{1}{r} = \frac{1}{q}$. Defining $r_1, \ldots,
r_n$ by $\frac{1}{p_j} + \frac{1}{r_j} = \frac{1}{q_j}$, $j = 1, \ldots, n$, it follows that $\frac{1}{r} =
\frac{1}{r_1} + \cdots +\frac{1}{r_n}$.
So $\ell_r = \ell_{r_1} \cdots \ell_{r_n}$. Given
$(\alpha_j)_{j=1}^\infty \in \ell_r$,  we write
$(\alpha_j)_{j=1}^\infty = (\alpha_j^1 \cdots
\alpha_j^n)_{j=1}^\infty$ where $(\alpha_j^k)_{j=1}^\infty \in
\ell_{r_k}, k= 1, \ldots, n.$ Since $(\alpha_j^k)_{j=1}^\infty \in
\ell_{r_k}$ and $(x_j^k)_{j=1}^\infty \in \ell_{p_k}^w(E_k)$ it
follows that $(\alpha_j^kx_j^k)_{j=1}^\infty \in \ell_{q_k}^w(E_k)$,
$k =1, \ldots, n$. Therefore
$$(\alpha_j \cdot A(x_j^1, \ldots, x_j^n))_{j=1}^\infty = (\alpha_j^1\cdots \alpha_j^n
\cdot A(x_j^1, \ldots, x_j^n))_{j=1}^\infty = (A(\alpha_j^1x_j^1, \ldots, \alpha_j^nx_j^n))_{j=1}^\infty \in \ell_q(F)$$
because $A$ is $(q;q_1, \ldots, q_n)$-summing. The identifications and embeddings we used are all isometric,
so the inequality between the norms follows.
\end{proof}

Using Proposition \ref{primerlema} and the inclusion $\ell_{1}^{w}(E)\subset
Rad(E)$ one obtains Defant-Voigt's result. Now combining (\ref{DV}) with the
inclusion theorem it is easy to prove that for $n\geq2$ and $\frac{1}{p_{1}%
}+\cdots+\frac{1}{p_{n}}\ge\frac1p$ one has
\begin{equation}
\label{DV2}{\mathcal{L}}(E_{1},\ldots,E_{n};\mathbb{K})=\Pi_{(p;p_{1},\ldots,p_{n}%
)}(E_{1},\ldots,E_{n};\mathbb{K}) {\rm ~whenever~}\frac{1}{p_{1}}%
+\cdots+\frac{1}{p_{n}}-\frac1p\ge n-1.
\end{equation}

Before start exploring the inclusion theorem we show that sometimes
the inclusion relationship turns out to be an equality. The next
result is simple (it appeared in essence in \cite[Theorem 16]{MT})
but indicates a good direction to be followed.

\begin{proposition}
Let $E_{1}, \ldots, E_{n}$ be cotype 2 spaces. Then
\[
\Pi_{(\frac{1}{n};1,\ldots,1)}(E_{1},\ldots,E_{n};F) = \Pi_{(\frac{2}%
{n};2,\ldots,2)}(E_{1},\ldots,E_{n};F)
\]
for every Banach space $F$.
\end{proposition}

\begin{proof}  It follows by combining (\ref{fact}) and the result saying that  if $E_j$ has cotype $2$ then
$\Pi_1(E_j;G_j)=\Pi_2(E_j;G_j)$ (see \cite[Corollary
11.16(a)]{Diestel}).
\end{proof}

We aim to prove a more general result for cotype 2 spaces:

\begin{theorem}
\label{cotipo} Let $1\le k\le n$ and assume that $E_{1}, \ldots, E_{k}$ have
cotype 2. If $p\le q$ and $1\le q_{i}\le2$, $i=1,\ldots,k$, satisfy that
$\sum_{i=1}^{k}\frac{1}{q_{i}}-\frac{1}{q} = k-\frac{1}{p}$ then
\[
\Pi_{(p;1,\ldots,1, p_{k+1},...,p_{n})}(E_{1},\ldots,E_{n};F) = \Pi
_{(q;q_{1},\ldots,q_{k}, p_{k+1},....,p_{n})}(E_{1},\ldots,E_{n};F)
\]
for every Banach space $F$.
\end{theorem}

\begin{proof} The inclusion
$$\Pi_{(p;1,\ldots,1, p_{k+1},...,p_n)}(E_{1},\ldots,E_{n};F) \subset \Pi_{(q;q_1,\ldots,q_k, p_{k+1},....,p_n)}(E_{1},\ldots,E_{n};F)$$
follows from the Inclusion Theorem.
Assume first that $q_i=2$ for $i=1,...,k$ and $A \in
\Pi_{(q_0;2,\ldots,2, p_{k+1},\ldots,p_n)}(E_{1},\ldots,E_{n};F)$
where $\frac{k}{2} + \frac{1}{q_0} = \frac{1}{p}$. Let
$(x_j^i)_{j=1}^\infty \in \ell_{1}^w(E_i)$ for $i = 1, \ldots, k$
and $(x_j^i)_{j=1}^\infty \in \ell_{p_i}^w(E_i)$ for $i = k+1,
\ldots, n$.  Since $E_i$ has cotype 2, by \cite[Proposition
6(a)]{AB1} we know that $\ell_1^w(E_i) = \ell_2 \cdot
\ell_2^w(E_i)$, $i = 1, \ldots, k$. Hence there are
$(\alpha_j^i)_{j=1}^\infty \in \ell_2$ and $(y_j^i)_{j=1}^\infty
\in \ell_2^w(E_k)$ such that $(x_j^i)_{j=1}^\infty = (\alpha_j^k
y_j^i)_{j=1}^\infty$, $i = 1, \ldots, k$. In this fashion,
$(\alpha_j^1 \cdots \alpha_j^k)_{j=1}^\infty \in \ell_2 \cdots
\ell_2 = \ell_{\frac{2}{k}}$ and $(A(y_j^1, \ldots,
y_j^k,x_j^{k+1},...,x_j^n))_{j=1}^\infty \in \ell_{q_0}(F)$. Since
$\frac{k}{2} + \frac{1}{q_0} = \frac{1}{p}$ it follows that
$$ (A(x_j^1, \ldots, x_j^n))_{j=1}^\infty =(\alpha_j^1 \cdots \alpha_j^k A(y_j^1, \ldots, y_j^k, x^{k+1}_j,\ldots,x^n_j))_{j=1}^\infty
\in \ell_p(F).$$
Now the general case follows again from the inclusion theorem, because
the assumption gives  that $\sum_{i=1}^n\frac{1}{q_i}
-\frac{1}{q}= \frac{k}{2}- \frac{1}{q_0}$ and then
$$
\Pi_{(q;q_1,\ldots,q_k,
p_{k+1},\ldots,p_n)}(E_{1},\ldots,E_{n};F)\subset
\Pi_{(q_0;2,\ldots,2, p_{k+1},\ldots,p_n)}(E_{1},\ldots,E_{n};F).$$
\end{proof}

\begin{corollary}
Let $1\le k\le n$. Assume that ${\mathcal{L}}(E_{1},\ldots,E_{k};F)
=\Pi_{(p;q_{1},\ldots,q_{k})}(E_{1},\ldots,E_{k};F)$ and that
$E_{k+1}, \ldots, E_{n} $ have cotype 2. If $p\le q$ and $1\le
q_{i}\le2$, $i=k+1,\ldots,n$, satisfy that
$\sum_{i=k+1}^{n}\frac{1}{q_{i}}-\frac{1}{q} = n-k-\frac{1}{p}$ then
\[
{\mathcal{L}}(E_{1},\ldots,E_{n};F)=\Pi_{(q;q_{1},\ldots,q_{n})}(E_{1}%
,\ldots,E_{n};F).
\]

\end{corollary}

\begin{proof}
From \cite[Corollary 3.2]{Port} if ${\mathcal
L}(E_1,\ldots,E_k;F)=\Pi_{(p;q_1,\ldots,q_k)}(E_1,\ldots,E_k;F)$
then ${\mathcal
L}(E_1,\ldots,E_n;F)=\Pi_{(p;q_1,\ldots,q_k,1\ldots,1)}(E_1,\ldots,E_n;F)$.
An application of Theorem \ref{cotipo} yields the result.
\end{proof}

\section{The role of the Littlewood-Orlicz property}\label{lop}

The aim of this section is to show how the Littlewood-Orlicz property can be used to obtain coincidence
results stronger than (\ref{DV2}). The proof of Theorem \ref{Teo3.1} will be also invoked in order to obtain
new coincidence results for $n$-linear functionals on the disc algebra and on the Hardy space ${\cal
H}^\infty$.

\begin{theorem}
\label{Teo3.1} Let $n$\ $\geq2,$ $1\leq p_{i}\leq\infty$, $p_{n}\geq2$ and
\[
n-\frac{3}{2}\leq\frac{1}{p_{1}}+\cdots+\frac{1}{p_{n}}.
\]
If $A\in{\mathcal{L}}(E_{1},\ldots,E_{n};\mathbb{K})$,
$E_{n}^{\prime}$ has the Littlewood-Orlicz property and
\[
n-\frac{3}{2}\leq\frac{1}{p_{1}}+\cdots+\frac{1}{p_{n}}-\frac{1}{p},
\]
then $\hat{A}\colon\ell_{p_{1}}^{w}(E_{1})\times\cdots\times\ell_{p_{n}}%
^{w}(E_{n})\rightarrow\ell_{p}$ is bounded. In other words, $\mathcal{L}%
(E_{1}, \ldots, ,E_{n};\mathbb{K})=\linebreak\Pi_{(p;p_{1}, \ldots,
p_{n})}(E_{1},\ldots, E_{n};\mathbb{K}).$ Moreover
$\Vert\hat{A}\Vert\leq\Vert A\Vert.$
\end{theorem}

\begin{proof} Denote by $A_{n-1}\colon E_{1}\times\cdots\times E_{n-1}\rightarrow E_{n}^{\prime }$ the
corresponding $(n-1)$-linear mapping defined by
\[
A_{n-1}(x^{1},\ldots,x^{n-1})(x^{n})=A(x^{1},\ldots,x^{n-1},x^{n}).
\]
One has, using the previous results that%
\[
\hat{A}_{n-1}\colon \ell_{1}^{w}(E_{1})\times\cdots\times\ell_{1}^{w}(E_{n-1}%
)\rightarrow\ell_{1}^{w}(E_{n}^{\prime})
\]
is bounded. In particular%
\[
\hat{A}_{n-1}\colon \ell_{1}^{w}(E_{1})\times\cdots\times\ell_{1}^{w}(E_{n-1}%
)\rightarrow\ell_{2}\otimes_{\pi}E_{n}^{\prime}%
\]
is bounded. Now we use a duality argument. Note that%
\begin{align*}
\hat{A}_{n-1}  &  \colon\ell_{1}^{w}(E_{1})\times\cdots\times\ell_{1}^{w}%
(E_{n-1})\rightarrow\ell_{1}^{w}(E_{n}^{\prime})\hookrightarrow\ell_{2}%
\otimes_{\pi}E_{n}^{\prime}\\
\hat{A}_{1}\left(  (x_{j}^{1})_{j},\ldots,(x_{j}^{n-1})_{j}\right)   & =\left(  A(x_{j}^{1}%
,\ldots,x_{j}^{n-1},\cdot)\right)  _{j}\hookrightarrow%
{\sum\limits_{j=1}^{\infty}}
e_{j}\otimes A(x_{j}^{1},\ldots,x_{j}^{n-1},\cdot)
\end{align*}
is bounded. We have for some suitable $\varepsilon_j$ that
\begin{align*}
\left\Vert \widehat{A}\left(  (x_{j}^{1})_{j},\ldots,(x_{j}^{n})_{j}\right)
\right\Vert _{1}  &  =%
\sum_j
\left\vert A(x_{j}^{1},\ldots,x_{j}^{n})\right\vert \\
&  =\left\vert
\sum_j
A(\varepsilon_{j}x_{j}^{1},\ldots,x_{j}^{n})\right\vert \\
&  =\left\vert
\sum_j
A(\varepsilon_{j}x_{j}^{1},\ldots,\frac{x_{j}^{n}}{\left\Vert (x_{j}^{n}%
)_{j}\right\Vert _{\ell_2^w(E_n)}})\right\vert \left\Vert (x_{j}^{n})_{j}\right\Vert
_{\ell_2^w(E_n)}\\
&  \leq\max_{\left\Vert (y_{j})_{j}\right\Vert _{\ell_2^w(E_n)}\leq1,\text{ }(y_{j}%
)_{j}\in\ell_{2}^{w}(E_{n})}\left\{  \left\vert
\sum_j
A(\varepsilon_{j}x_{j}^{1},\ldots,y_{j})\right\vert \left\Vert (x_{j}^{n}%
)_{j}\right\Vert _{\ell_2^w(E_n)}\right\} \\
&  \leq\max_{\left\Vert u\right\Vert \leq1,u\in\mathcal{L}(\ell_{2};E_{n}%
)}\left\{  \left\vert
\sum_j
A(\varepsilon_{j}x_{j}^{1},\ldots,u(e_{j}))\right\vert \left\Vert (x_{j}^{n}%
)_{j}\right\Vert _{\ell_2^w(E_n)}\right\}  =(\ast).
\end{align*}
In this last inequality we used the identification:
\begin{align*}
\ell_{2}^{w}(E_{n})  &  \longleftrightarrow \mathcal{L}(\ell_{2};E_{n})\\
(y_{j})_{j}  &  \longleftrightarrow  T_{(y_{j})_{j}}%
\end{align*}
given by $T_{(y_{j})_{j}}((z_{j})_{j})=\sum_{j}y_{j}z_{j}.$
Now, using the inclusion%
\begin{align*}
\mathcal{L}(\ell_{2};E_{n})  &  \hookrightarrow\left(
\ell_{2}\otimes_{\pi
}E_{n}^{\prime}\right)  ^{\prime}\\
u  &  \rightarrow\varphi\colon \ell_{2}\otimes_{\pi}E_{n}^{\prime}\rightarrow
\mathbb{K}\\
\varphi\left(  (\lambda_{j})_{j}\otimes x^{\prime}\right)   &  =x^{\prime }(  u((\lambda_{j})_{j}))
\end{align*}
and the identification%
\begin{align*}
\mathcal{L}(\ell_{2};E_{n}^{\prime\prime})  &  =\left(
\ell_{2}\otimes_{\pi
}E_{n}^{\prime}\right)  ^{\prime}\\
S  &  \rightarrow\varphi_{S}\colon \ell_{2}\otimes_{\pi}E_{n}^{\prime}%
\rightarrow\mathbb{K}\\
\varphi_{S}(x\otimes y)  &  =S(x)(y)
\end{align*}
we get%
\begin{align*}
(\ast)  &  =\max_{\left\Vert u\right\Vert \leq1,u\in\mathcal{L}(\ell_{2}%
;E_{n})}\left\{  \left\vert
\sum_j
A(\varepsilon_{j}x_{j}^{1},\ldots,x_{j}^{n-1},u(e_{j}))\right\vert \left\Vert
(x_{j}^{n})_{j}\right\Vert _{\ell_2^w(E_n)}\right\} \\
&  \leq\max_{\left\Vert \varphi\right\Vert \leq1,\varphi\in\left(  \ell
_{2}\otimes_{\pi}E_{n}^{\prime}\right)  ^{\prime}}\left\{  \left\vert
\sum_j
\varphi\left(  e_{j}\otimes A(\varepsilon_{j}x_{j}^{1},\ldots,x_{j}%
^{n-1},\cdot)\right)  \right\vert \left\Vert (x_{j}^{n})_{j}\right\Vert
_{\ell_2^w(E_n)}\right\} \\
&  =\max_{\left\Vert \varphi\right\Vert \leq1,\varphi\in\left(  \ell
_{2}\otimes_{\pi}E_{n}^{\prime}\right)  ^{\prime}}\left\{  \left\vert
\varphi\left(
\sum_j
e_{j}\otimes A(\varepsilon_{j}x_{j}^{1},\ldots,x_{j}^{n-1},\cdot)\right)  \right\vert
\left\Vert (x_{j}^{n})_{j}\right\Vert _{\ell_2^w(E_n)}\right\} \\
&  =\left\Vert
\sum_j
e_{j}\otimes A(\varepsilon_{j}x_{j}^{1},\ldots,x_{j}^{n-1},\cdot)\right\Vert
_{\ell_{2}\otimes_{\pi}E_{n}^{\prime}}\left\Vert (x_{j}^{n})_{j}\right\Vert
_{\ell_2^w(E_n)}\\
&  \overset{(\ast\ast)}{\leq}C\left\Vert \left(  A(\varepsilon_{j}x_{j}%
^{1},\ldots,x_{j}^{n-1},\cdot)\right)  _{j}\right\Vert _{\ell_{1}^{w}(E_{n}^{\prime
})}\left\Vert (x_{j}^{n})_{j}\right\Vert _{\ell_2^w(E_n)}\\
&  =C\left\Vert \hat{A}_{1}\left(  (\varepsilon_{j}x_{j}^{1})_{j}%
,\ldots,(x_{j}^{n-1})_{j}\right)  \right\Vert _{\ell_{1}^{w}(E_{n}^{\prime}%
)}\left\Vert (x_{j}^{n})_{j}\right\Vert _{\ell_2^w(E_n)}\\
&  \leq C\left\Vert \hat{A}_{1}\right\Vert \left\Vert (x_{j}^{1}%
)_{j}\right\Vert _{\ell_1^w(E_1)}\ldots\left\Vert (x_{j}^{n-1})_{j}\right\Vert
_{\ell_1^w(E_{n-1})}\left\Vert (x_{j}^{n})_{j}\right\Vert _{\ell_2^w(E_n)}<\infty,
\end{align*}
where in (**) we used that the inclusion
$\ell_{1}^{w}(E_{n}^{\prime
})\hookrightarrow\ell_{2}\otimes_{\pi}E_{n}^{\prime}$ is
continuous. We have just proved that
$$\hat A \colon \ell_1^w(E_1) \times \cdots \times \ell_1^w(E_{n-1}) \times \ell_2^w(E_n) \rightarrow \ell_1 $$
is bounded. The proof is completed by using the Inclusion Theorem.
\end{proof}

Taking into account the inclusion $Rad(E)\subset\ell_{2}^{w}(E)$,
our aim is now to analyze when the result in Proposition
\ref{primerlema} can be lifted to $\ell_{2}^{w}(E_{i})$. In other
words, when $\mathcal{L}(E_{1}, \ldots, ,E_{n})=\Pi_{(1;2,
\ldots,2)}(E_{1},\ldots, E_{n}).$ In this direction D.
P\'erez-Garc\'{\i}a proved the following result:

\begin{theorem}
{\rm \cite[Corollary 2.5]{david}} Let $E_{1}, \ldots, E_{n}$ be
$\mathcal{L}_{\infty}$-spaces. If $A \colon E_{1} \times\cdots\times
E_{n} \longrightarrow\mathbb{K}$ is multilinear and bounded, then
$\hat A \colon \ell_{2}^{w}(E_{1})
\times\cdots\times\ell_{2}^{w}(E_{n}) \longrightarrow \ell_{1}$ is
also bounded. In other words, $\mathcal{L}(E_{1}, \ldots
,E_{n};\mathbb{K})=\Pi_{(1;2, \ldots,2)}(E_{1},\ldots,
E_{n};\mathbb{K}).$
\end{theorem}

Applying the idea used in the proof of Theorem \ref{Teo3.1} we can reprove the
result above and generalize it to a larger class of spaces. Recall that a
bilinear form $A\colon E_{1}\times E_{2}\longrightarrow\mathbb{K}$ is
$2$-dominated if and only if it is absolutely $(1;2,2)$-summing, i.e., if and
only if $\hat{A}\colon\ell_{2}^{w}(E_{1})\times\ell_{2}^{w}(E_{2}%
)\longrightarrow\ell_{1}$ is bounded.

\begin{theorem}
\label{pggen} Let $n\geq2$ and $E_{1},\ldots,E_{n}$ be Banach spaces such that
$E_{3}^{\prime},\ldots,E_{n}^{\prime}$ have the Littlewood-Orlicz property and
every continuous bilinear form on $E_{1}\times E_{2}$ is 2-dominated. If
$A\colon E_{1}\times\cdots\times E_{n}\longrightarrow\mathbb{K}$ is
multilinear and bounded, then $\hat{A}\colon\ell_{2}^{w}(E_{1})\times
\cdots\times\ell_{2}^{w}(E_{n})\longrightarrow\ell_{1}$ is also bounded. In
other words, $\mathcal{L}(E_{1}, \ldots,E_{n};\mathbb{K})=\Pi_{(1;2, \ldots,2)}%
(E_{1},\ldots, E_{n};\mathbb{K}).$
\end{theorem}

\begin{proof} We proceed by induction on $n$. The case $n=2$ follows by assumption. Assume that the result holds for
$n\geq 2$. Let $B\colon E_{1}\times\cdots\times E_{n}\longrightarrow\mathbb{K}$ be given. By the induction
hypothesis $\hat
{B}\colon\ell_{2}^{w}(E_{1})\times\cdots\times\ell_{2}^{w}(E_{n}%
)\longrightarrow\ell_{1}$ is bounded. It follows that for every Banach space $F$ and every $C\colon
E_{1}\times\cdots\times E_{n}\longrightarrow F$, the mapping
$\hat{C}\colon\ell_{2}^{w}(E_{1})\times\cdots\times\ell_{2}^{w}(E_{n}%
)\longrightarrow\ell_{1}^{w}(F)$ is bounded. Given $A\colon
E_{1}\times \cdots\times E_{n+1}\longrightarrow\mathbb{K}$,
defining $A_{n}\colon E_{1}\times\cdots\times E_{n}\longrightarrow
E_{n+1}^{\prime}$ in the obvious way, we have that
$\hat{A}_{n}\colon\ell_{2}^{w}(E_{1})\times\cdots\times
\ell_{2}^{w}(E_{n})\longrightarrow\ell_{1}^{w}(E_{n+1}^{\prime})$
is bounded. Since $E_{n+1}^{\prime}$ has the Littlewood-Orlicz
property we have that
$\hat{A}_{n}\colon\ell_{2}^{w}(E_{1})\times\cdots\times\ell_{2}^{w}%
(E_{n})\longrightarrow\ell_{2}\otimes_{\pi}E_{n+1}^{\prime}$ is bounded. Using the duality argument from the
proof of Theorem \ref{Teo3.1} it follows that
$\hat{A}\colon\ell_{2}^{w}(E_{1})\times\cdots\times\ell_{2}^{w}(E_{n+1}%
)\longrightarrow\ell_{1}$ is bounded as well.\end{proof}

\begin{theorem}
Let $E_{1},\ldots,E_{n}$ be Banach spaces such that $E_{1}=E_{2}$
and each $E_{j}$ is either an $\mathcal{L}_{\infty}$-space, the disc
algebra $\mathcal{A}$ or the Hardy space $\mathcal{H}^{\infty}$.
{Then $\mathcal{L}(E_{1}, \ldots, E_{n};\mathbb{K})=\Pi_{(1;2,
\ldots,2)}(E_{1},\ldots, E_{n};\mathbb{K}).$}
\end{theorem}

\begin{proof} First we need to know that the duals of an $\mathcal{L}_{\infty}%
$-space, the disc algebra $\mathcal{A}$ and the Hardy space $\mathcal{H}%
^{\infty}$ have the Littlewood-Orlicz property.
(i) It is well known that the dual of a $\mathcal{L}_{\infty}$-space has the
Littlewood-Orlicz property.
(ii) The dual $\mathcal{A}^{\prime}$ of the disc algebra is a G.T. space
\cite[Corollary 2.7]{bourgain} and has cotype $2$ \cite[Corollary
2.11]{bourgain}, hence $\mathcal{A}^{\prime}$ has the Littlewood-Orlicz property.
(iii) From \cite[Theorem 6.17]{Pi} (and using the notation from \cite{Pi}) we
know that %
$
L_{1}/\overline{H_{0}^{1}}%
$
is a GT space of cotype $2.$ From \cite[Proposition 6.2]{Pi} we know that %
$
\left(  L_{1}/\overline{H_{0}^{1}}\right)  ^{\prime\prime}%
$ is a GT space. Since
\[
\mathcal{H}^{\infty}=\left(  L_{1}/\overline{H_{0}^{1}}\right)
^{\prime}\text{
\cite[Remark, page 84]{Pi},}%
\]
it follows that $\left(  \mathcal{H}^{\infty}\right)  ^{\prime}$ is
a GT space. It is well known that a Banach space has the same cotype
of its bidual. So,
\[
\left(  \mathcal{H}^{\infty}\right)  ^{\prime}=\left(  L_{1}/\overline{H_{0}%
^{1}}\right)  ^{\prime\prime}%
\]
has cotype $2$. It follows that $\left(  \mathcal{H}^{\infty}\right)
^{\prime}$ has the Littlewood-Orlicz property. The fact that
bilinear forms on either an $\mathcal{L}_{\infty}$-space or the disc
algebra or the Hardy space are $2$-dominated was proved in
\cite[Theorem 3.3]{bote97} and \cite[Proposition 2.1]{pams},
respectively.
\end{proof}

The same reasoning gives the following result:

\begin{proposition}
If $E_{2}^{\prime},\ldots,E_{n}^{\prime}$ have the Littlewood-Orlicz
property and $A\colon E_{1}\times\cdots\times
E_{n}\longrightarrow\mathbb{K}$ is multilinear and bounded, then
$\mathcal{L}(E_{1}, \ldots ,E_{n};\mathbb{K})=\Pi_{(1;1,2,
\ldots,2)}(E_{1},\ldots, E_{n};\mathbb{K}).$
\end{proposition}

\section{From bilinear to multilinear mappings}

\label{1r}

In the previous section, when $E$ is an
$\mathcal{L}_{\infty}$-space, the disc algebra $\mathcal{A}$ or the
Hardy space $\mathcal{H}^{\infty},$ using that
$\mathcal{L}(^{2}E;\mathbb{K})=\Pi_{(1;2,2)}(^{2}E;\mathbb{K})$, we
have shown that
\[
\mathcal{L}(^{n}E;\mathbb{K})=\Pi_{(1;2,\ldots,2)}(^{n}E;\mathbb{K})
\]
for every $n>2.$ Although the lift of bilinear results to
multilinear results is not a straightforward step in general, in the
present section we obtain a general argument showing
how bilinear coincidences of the type $\mathcal{L}(^{2}%
E;\mathbb{K})=\Pi_{(1;r,r)}(^{2}E;\mathbb{K})$ can generate coincidences for $n$-linear forms, $n\geq3$.


\begin{definition}
\textrm{Let $\frac{1}{q_{1}}+\frac{1}{q_{2}}\cdots+\frac{1}{q_{n}}\geq\frac
{1}{p}.$ We say that $A\in\mathcal{L}(E_{1},\ldots,E_{n};F)$ is \textit{weakly
$(p;q_{1},\ldots,q_{n})$-summing} if $(A(x_{j}^{1},\ldots,x_{j}^{n}))_{j}%
\in\ell_{p}^{w}(F)$ whenever $(x_{j}^{k})_{j}\in\ell_{p}^{w}(E_{k}),$
$k=1,\ldots,n$. The space formed by these mappings is denoted by
$\Pi_{w(p;q_{1},\ldots,q_{n})}(E_{1},\ldots,E_{n};F)$ and the norm
$\pi_{w(p;q_{1},\ldots,q_{n})}$ is defined in the natural way. }
\end{definition}

Next Lemma is simple but useful:

\begin{lemma}
\label{was} Let $n\in{\mathbb{N}}$ and let $E_{1},\ldots,E_{n}$ be Banach
spaces. The following are equivalent:

(i)
$\mathcal{L}(E_{1},\ldots,E_{n};\mathbb{K})=\Pi_{(p;q_{1},\ldots,q_{n})}
(E_{1},\ldots,E_{n};\mathbb{K})$ and
$\pi_{(p;q_{1},\ldots,q_{n})}\leq C\Vert\cdot\Vert.$

(ii) $\mathcal{L}(E_{1},\ldots,E_{n};F)=\Pi_{w(p;q_{1},\ldots,q_{n})}
(E_{1},\ldots,E_{n};F)$ for every Banach space $F$ and $\pi_{w(p;q_{1}%
,\ldots,q_{n})}\leq C\Vert\cdot\Vert.$

(iii) There exists $C>0$ such that
\[
\Vert(x_{j}^{1}\otimes\cdots\otimes x_{j}^{n})_{j} \Vert_{\ell_{p}^{w}%
(E_{1}\otimes_{\pi}\cdots\otimes_{\pi}E_{n})}\leq C\prod_{i=1}^{n}\Vert
(x_{j}^{i})_{j}\Vert_{\ell_{q_{i}}^{w}(E_{i})}%
\]
for all $(x_{j}^{i})_{j}\in\ell_{q_{i}}^{w}(E_{i})$, $i=1,\ldots,n$.
\end{lemma}

\begin{proof} $(i)\Longrightarrow(ii)$ Let $A\in\mathcal{L}(E_{1},\ldots,E_{n};F)$. If
$(x_{j}^{k})_{j}\in \ell_{q_{k}}^{w}(E_{k}),$ $k=1,\ldots,n,$ then%
\begin{align*}
\sup_{\varphi\in B_{F^{\prime}}}\left( \sum_j \left\vert \varphi(A(x_{j}^{1},\ldots,x_{j}^{n}))\right\vert
^{p}\right)  ^{1/p} &  \leq\sup_{\varphi\in B_{F^{\prime}}}\pi_{(p;q_1, \ldots,q_n)} (\varphi\circ
A)\left\Vert (x_{j}^{1})_{j}\right\Vert _{\ell^w_{q_{1}}(E_1)}\cdots\left\Vert
(x_{j}^{n})_{j}\right\Vert _{\ell^w_{q_{n}}(E_n)}\\
&  \leq\sup_{\varphi\in B_{F^{\prime}}}C\left\Vert \varphi\circ A\right\Vert  \left\Vert
(x_{j}^{1})_{j}\right\Vert _{\ell^w_{q_{1}}(E_1)
}\cdots\left\Vert (x_{j}^{n})_{j}\right\Vert _{\ell^w_{q_{n}}(E_n)}\\
&  \leq C\left\Vert A\right\Vert \left\Vert
(x_{j}^{1})_{j}\right\Vert
_{\ell^w_{q_{1}}(E_1)}\cdots\left\Vert (x_{j}^{n})_{j}\right\Vert _{\ell^w_{q_{n}}(E_n)}.%
\end{align*}
\noindent $(ii)\Longrightarrow(iii)$ Take
$F=E_{1}\otimes_{\pi}\cdots\otimes_{\pi}E_{n}$ and $A\colon
E_{1}\times\cdots\times E_{n}\rightarrow
E_{1}\otimes_{\pi}\cdots\otimes _{\pi}E_{n}$ given by
$A(x_{1},\ldots,x_{n})=x_{1}\otimes\cdots\otimes x_{n}$.\\
\noindent $(iii)\Longrightarrow(i)$ Given $A\in\mathcal{L}(E_{1},\ldots,E_{n};\mathbb{K}%
)$, its linearization $T \colon E_{1}\otimes_{\pi}\cdots\otimes_{\pi}E_{n}%
\rightarrow\mathbb{K}$ is bounded and then $\tilde{T}\colon \ell_{p}^{w}%
(E_{1}\otimes_{\pi}\cdots\otimes_{\pi}E_{n})\rightarrow\ell_{p}$ is bounded. Now
\begin{eqnarray*}
\Vert(A(x_{j}^{1},\ldots,x_{j}^{n}))_j\Vert_{p}& =&\Vert(T(x_{j}^{1}\otimes\cdots\otimes
x_{j}^{n}))_j\Vert_{p}\\
&\leq & \Vert T\Vert\Vert(x_{j}^{1}\otimes\cdots\otimes x_{j}
^{n})_j\Vert_{\ell^w_p(E_1\otimes_\pi\cdots\otimes_\pi E_n)}\\
&\leq &C\Vert T\Vert\prod_{i=1}^{n}\Vert(x_{j}^{i})_j\Vert_{\ell^w_{q_i}(E_i)}.
\end{eqnarray*}
\end{proof}

\begin{theorem}
\label{gen}Let $1\leq r\leq2$. If
$\mathcal{L}(^{2}E;\mathbb{K})=\Pi_{(1;r,r)}(^{2}E;\mathbb{K})
\label{30D}$ and $\pi_{(1;r,r)}\leq C\Vert\cdot\Vert$, then

\begin{enumerate}
\item[(i)] For $n$ even, $\mathcal{L}(^{n}E;\mathbb{K})=\Pi_{(1;r,\ldots,r)}(^{n}E;\mathbb{K}) $ and
$\pi_{(1;r,\ldots,r)}\leq C^{n/2}\Vert\cdot\Vert.$

\item[(ii)] For $n\geq3$ and odd, $\mathcal{L}(^{n}E;\mathbb{K})=\Pi_{(r;r,\ldots
,r)}(^{n}E;\mathbb{K}) $ and $\pi_{(r;r,\ldots,r)}\leq
C^{(n-1)/2}\Vert\cdot\Vert.$
\end{enumerate}
\end{theorem}

\begin{proof} (i) Let $n=2m$, $m\in{\mathbb{N}}$, and $A\in\mathcal{L}%
(^{2m}E;\mathbb{K}).$ Using the associativity of the projective norm $\pi$ it is easy to see that there is an
$m$-linear mapping $B\in\mathcal{L}(^{m}(E \hat\otimes_{\pi}E);\mathbb{K})$ such that
\[
B(x^{1}\otimes x^{2},\ldots,x^{2m-1}\otimes x^{2m})=A(x^{1},x^{2}%
,\ldots,x^{2m-1},x^{2m}).
\]
Using Defant-Voigt Theorem and Lemma \ref{was} we get%
\begin{align*}%
\lefteqn{\sum_j
\left\vert A(x_{j}^{1},\ldots,x_{j}^{2m})\right\vert    } \\
& &  & = \sum_j
\left\vert B(x_{j}^{1}\otimes x_{j}^{2},\ldots,x_{j}^{2m-1}\otimes x_{j}%
^{2m})\right\vert \\
& & & \leq \pi_{(1;1,\ldots,1)}( B)\left\Vert (x_{j}^{1}\otimes
x_{j}^{2})_{j}\right\Vert _{\ell^w_1(E\otimes_\pi E)}\cdots\left\Vert (x_{j}^{2m-1}\otimes x_{j}%
^{2m})_{j}\right\Vert _{\ell^w_1(E_{}\otimes_\pi E_{})}\\
& & & \leq\left\Vert B\right\Vert \left(  C\left\Vert (x_{j}^{1}%
)_{j}\right\Vert _{\ell^w_r(E)}\left\Vert
(x_{j}^{2})_{j}\right\Vert _{\ell^w_r(E)}\right) \cdots\left(
C\left\Vert (x_{j}^{2m-1})_{j}\right\Vert
_{\ell^w_r(E_{})}\left\Vert
(x_{j}^{2m})_{j}\right\Vert _{\ell^w_r(E_{})}\right) \\
& & & = C^{m}\left\Vert A\right\Vert \left\Vert (x_{j}^{1})_{j}\right\Vert
_{\ell^w_r(E)}\cdots\left\Vert (x_{j}^{2m})_{j}\right\Vert _{\ell^w_r(E_{})}.%
\end{align*}
(ii) Let $n=2m+1$, $m\in{\mathbb{N}}$, and $A\in\mathcal{L}(^{2m+1}%
E;\mathbb{K}).$ From (i) and \cite[Corollary 3.2]{Port} we conclude that %
$ A\in\Pi_{(1;r,\ldots,r,1)}(^{2m+1}E;\mathbb{K})$
and it is not difficult to check that %
$ \pi_{(1;r,\ldots,r,1)}\leq C^{m}\Vert \cdot \Vert$. Using the
Inclusion Theorem we conclude that
$A\in\Pi_{(p;r,\ldots,r,p)}(^{2m+1} E;\mathbb{K})$ for any $1\leq
p<\infty$. The result is now finished.
\end{proof}

Let us point out some connection of Littlewood-Orlicz property on
$E^{\prime}$ and
$\mathcal{L}(^{2}E;\mathbb{K})=\Pi_{(1;r,r)}(^{2}E;\mathbb{K})$.

\begin{proposition}
\label{lo}Let $E$ be a Banach space. The following statements are equivalent.

(i) $E^{\prime}$ has the Littlewood-Orlicz property.

(ii) ${\mathcal{L}}(X,E;\mathbb{K})=\Pi_{(1;1,2)}(X,E;\mathbb{K})$
for any Banach space $X$.

(iii) $\ell_{1}^{w}(X)\otimes_{\pi}\ell_{2}^{w}(E)\subset\ell_{1}^{w}%
(X\otimes_{\pi}E).$
\end{proposition}

\begin{proof}$(i) \Longrightarrow (ii)$   Let $A \colon X\times
E\to {\mathbb K}$ be a bounded bilinear form and let $T_A \colon X\to E'$ be the corresponding linear
operator. Assume that $(x_j)_j\in \ell_1^w(X)$ and $(y_j)_j\in \ell_2^w(E)$.
\begin{eqnarray*}
\sum_{j}|A(x_j,y_j)|&=& \sum_{j}| T_A(x_j)(y_j)|\\
&=& \sup_{|\alpha_j|=1}|\sum_{j} T_A(x_j)(\alpha_jy_j)|\\
&\le& \Vert( T_A(x_j))_{j}\Vert_{\ell_2\otimes_\pi (E')}\Vert(y_j)_j\Vert_{\ell_2^w(E)}\\
&\le & C\Vert( T_A(x_j))_{j}\Vert_{\ell_1^w(E')}\Vert(y_j)_j\Vert_{\ell_2^w(E)}\\
&\le & C\Vert A\Vert\Vert (x_j)_{j}\Vert_{\ell_1^w(X)}\Vert(y_j)_j\Vert_{\ell_2^w(E)}.
\end{eqnarray*}
$(ii) \Longrightarrow (i)$  Let $(x'_j)_j\in \ell_1^w(E')$ be given. Consider the bounded bilinear
map $A:c_0\times E\to {\mathbb K}$ defined by the condition $A(e_j,x)= x'_j(x)$ for $x\in E$.
To show that $(x'_j)_j\in \ell_2\otimes_\pi (E')$ it suffices to
see that
$$\sum_j | x'_j(x_j)|\le C \Vert
(x_j)_j\Vert_{\ell_2^w(E)}$$ and, using $X=c_0$ in the assumption,
this follows using that
$$\sum_j | x'_j(x_j)|= \sum_j |A( e_j,x_j)|\le \|A\| \Vert(e_j)_j\|_{\ell_1^w(c_0)} \Vert
(x_j)_j\Vert_{\ell_2^w(E)}.$$
$(ii) \Longleftrightarrow (iii)$ It is a particular case in Lemma \ref{was}.
\end{proof}

The same idea used in the proof of Theorem \ref{gen} provides the following
slight improvement:

\begin{theorem}
Let $n$ be a positive integer. For $i=1,\ldots,2n+1$ let $E_{i}$ be a Banach
space and $1\leq r_{2n+1}\leq r_{1},\ldots,r_{2n}\leq2$. If
\[
\mathcal{L}(E_{1},E_{2};\mathbb{K})=\Pi_{(1;r_{1},r_{2})}(E_{1},E_{2};\mathbb{K})\ \text{and}%
\ \pi_{(1;r_{1},r_{2})}\leq C_{2}\Vert\cdot\Vert,
\]%
\[
\mathcal{L}(E_{3},E_{4};\mathbb{K})=\Pi_{(1;r_{3},r_{4})}(E_{3},E_{4};\mathbb{K})\ \text{and}%
\ \pi_{(1;r_{3},r_{4})}\leq C_{4}\Vert\cdot\Vert,\,\,\ldots
\]%
\[
\mathcal{L}(E_{2n-1},E_{2n};\mathbb{K})=\Pi_{(1;r_{2n-1},r_{2n})}(E_{2n-1},E_{2n};\mathbb{K}%
)\ \text{and}\ \pi_{(1;r_{2n-1},r_{2n})}\leq C_{2n}\Vert\cdot\Vert,
\]
then
\[
\mathcal{L}(E_{1},\ldots,E_{2n};\mathbb{K}) =\Pi_{(1;r_{1},\ldots,r_{2n})}(E_{1}%
,\ldots,E_{2n};\mathbb{K}) {\ ~and~} \pi_{(1;r_{1},\ldots,r_{2n})}
\leq C_{2}\cdots C_{2n}\Vert\cdot\Vert,
\]
\[
\mathcal{L}(E_{1},\ldots,E_{2n+1};\mathbb{K}) =\Pi_{(r_{2n+1};r_{1},\ldots,r_{2n+1}%
)}(E_{1},\ldots,E_{2n+1};\mathbb{K}) ~and~
\pi_{(r_{2n+1};r_{1},\ldots,r_{2n+1})}\leq C_{2}\cdots
C_{2n}\Vert\cdot\Vert.
\]


\end{theorem}

\section{The role of almost summing mappings}

\label{almost}

Let $n\geq2$, $A\in\mathcal{L}(E_{1},\ldots,E_{n};F)$ and $1 \leq k \leq n$.
Recall that the $k$-linear mapping $A_{k}$ is defined by
\[
A_{k} \colon E_{1} \times\cdots\times E_{k} \rightarrow\mathcal{L}%
(E_{k+1},\ldots,E_{n};F)~,~A_{k}(x_{1}, \ldots, x_{k})(x_{k+1},\ldots, x_{n})
= A(x_{1}, \ldots, x_{n}) .
\]

We first mention several connections between absolutely summing and
almost summing multilinear mappings. Clearly $\Pi_{a.s}(E_{1},...,E_{n};F)$ coincides with $\Pi_{(2;2,...2)}%
(E_{1},...,E_{n};F)$ whenever $F$ is a Hilbert space because $Rad(F)=\ell
_{2}(F)$, and the corresponding inclusions hold whenever $F$ has type $p$ or
cotype $q$.

In the linear case one has (see \cite{Diestel}) $\bigcup_{p>0}\Pi_{p}%
(E;F)\subset\Pi_{a.s}(E;F)$. Using this linear containment
relationship and (\ref{fact}) - see also \cite{Nach} - it is not
difficult to see that this relationship also holds for $p$-dominated
multilinear maps, i.e.
\[
\bigcup_{p>0}\Pi_{(p/n;p\ldots,p)}(E_{1},\ldots, E_{n};F)\subset\Pi_{a.s}%
(E_{1},\ldots, E_{n};F).
\]

\begin{proposition}
\label{corgt} Let $A\in\mathcal{L}(E_{1},\ldots,E_{n};\mathbb{K})$ and $A_{n-1}%
\in\mathcal{L}(E_{1},\ldots,E_{n-1}; E_{n}^{\prime})$.

(i) If $A\in\Pi_{(1;2,\ldots,2)}(E_{1},\ldots,E_{n};\mathbb{K})$
then $A_{n-1}\in \Pi_{a.s}(E_{1},\ldots,E_{n-1}; E_{n}^{\prime})$

(ii) If $E_{n}^{\prime}$ is a $GT$-space of cotype 2 and
$A_{n-1}\in\Pi _{a.s}(E_{1},\ldots,E_{n-1}; E_{n}^{\prime})$ then
$A\in\Pi_{(1;2,\ldots ,2)}(E_{1},\ldots,E_{n};\mathbb{K})$.
\end{proposition}

\begin{proof}
(i) Assume $A\in \Pi_{(1;2,\ldots,2)}(E_1,\ldots,E_{n};\mathbb{K})$.
Using that $\ell_2\otimes_\pi F\subset Rad(F)$ one has
\begin{eqnarray*}\Vert (
A_{n-1}(x_j^1,\ldots,x_j^{n-1}))_j\Vert_{Rad(E_n')}&\le & C\Vert
(A_{n-1}(x_j^1,\ldots,x_j^{n-1}))_j\Vert_{\ell_2\otimes_\pi
E'_n}\\&=&\sup_{\|(x_j^n)_j\|_{\ell_2^w(E_n)}=1} |\sum_j  A_{n-1}
(x_j^1,\ldots,x_j^{n-1})(x_j^n)|\\
&\le & \pi_{(1;2,\ldots,2)}(A)\prod_{i=1}^{n-1}\|(x_j^i)_j\|_{\ell_2^w(E_i)}.
\end{eqnarray*}
(ii)Assume that $A_{n-1}\in \Pi_{a.s}(E_1,\ldots,E_{n-1}; E_n')$.
From (\ref{GT}) one has $\ell_2\otimes_\pi E'_n= Rad(E_n')$. Hence
we obtain, for any $|\alpha_j|=1$,
\begin{eqnarray*} \sum_j A(\alpha_jx_j^1,\ldots,x_j^n)
 & = & \sum_j A_{n-1}(\alpha_jx_j^1,\ldots,x_j^{n-1})(x_j^n)
\\ & \le & \Vert
( A_{n-1}(\alpha_jx_j^1,\ldots,x_j^{n-1}))_j\Vert_{\ell_2\otimes_\pi
E'_n}\|(x_j^n)_j\|_{\ell_2^w(E_n)}\\
& \le & C\Vert (
A_{n-1}(\alpha_jx_j^1,\ldots,x_j^{n-1}))_j\Vert_{Rad(E_n')}\|(x_j^n)_j\|_{\ell_2^w(E_n)}\\
& \le &C\|A_{n-1}\|_{a.s}\prod_{i=1}^n\|(x_j^i)_j\|_{\ell_2^w(E_i)}
\end{eqnarray*}
\end{proof}

\begin{theorem}
\label{as} Let $1\le k<n$ and
$A\in{\mathcal{L}}(E_{1},...,E_{n};\mathbb{K})$ be such that
\[
A_{k}\in\Pi_{a.s}(E_{1},...,E_{k};{\mathcal{L}}(E_{k+1},..,E_{n};\mathbb{K})).
\]
Then,
\[
\hat A:\ell_{2}^{w}(E_{1})\times...\times\ell_{2}^{w}(E_{k})\times
Rad(E_{k+1}) \times\cdots\times Rad(E_{n}) \to\ell_{1}%
\]
is bounded. Moreover $\|\hat A\|\le\|A_{k}\|_{a.s}.$
\end{theorem}

\begin{proof} Let $(x_{j}^{i})_{j}$ be a finite sequence in $E_{i}$ for $i=1,\ldots
,n.$ Take a scalar sequence $(\alpha_{j})_{j}$, denote $A^j_k=
A_k(x_j^1,x_j^2,...,x_j^k)$ and define
$$f_\alpha(t_k)=\sum_j\alpha_j A^j_kr_j(t_k);~f_i(t_i)=\sum_j r_j(t_i)x_j^i,~ i=k+1,\ldots,n-1;~{\rm and}$$
$$f_{n}(t_{k},\ldots,t_{n-1})=\sum_j r_j(t_{k})\cdots r_j(t_{n-1})x_j^n, \quad t_{k},\ldots,t_{n-1} \in
[0,1].$$
The orthogonality of the
Rademacher system  shows that
\begin{eqnarray*}
\lefteqn{\sum_j A(\alpha_jx_j^1,\ldots,x_j^n)}\\ &=&\sum_j
A_k(\alpha_jx_j^1,\ldots,x_j^k)(x_j^{k+1},\cdots,x_j^n) \\ & = & \sum_j \alpha_j
A_k^j(x_j^{k+1},\cdots,x_j^n)
\\&=&\int_0^1\cdots\int_0^1
f_\alpha(t_k)( f_{k+1}(t_{k+1}),\ldots, f_{n-1}(t_{n-1}),f_n(t_{k},\ldots,t_{n-1}))dt_k \cdots dt_{n-1}\\
&\le&\int_0^1\cdots\int_0^1 \Big(\int_0^1\| f_\alpha(t_k)\|
\|f_n(t_{k},\ldots,t_{n-1})\|dt_k\Big) \|
f_{k+1}(t_{k+1})\|\cdots \|f_{n-1}(t_{n-1})\|\  dt_{k+1}\cdots dt_{n-1} \\
&\le&\|A_k\|_{a.s}\prod_{i=1}^k\|(x_j^i)_j\|_{\ell_2^w(E_i)}\\
&.&\int_0^1\cdots\int_0^1 \Big(\int_0^1 \|f_{n}(t_k,...,t_{n-1})\|^2dt_{k}\Big)^{1/2}
\| f_{k+1}(t_{k+1})\|\cdots \|f_{n-1}(t_{n-1})\| dt_{k+1}\cdots dt_{n-1} \\&\le&   \|A_k\|_{a.s}
\prod_{i=1}^k\|(x_j^i)_j\|_{\ell_2^w(E_i)}
\Big(\prod_{i=k+1}^{n-1}\|(x_j^i)_j\|_{Rad_1}\Big)\|( x_j^n)_j\|_{Rad_2}.
\end{eqnarray*}
This allows to conclude the proof.
\end{proof}

Let us see that Theorem \ref{as} has nice consequences.

\begin{theorem}
\label{t2} If $1\le p\le2$ and $E_{2}^{\prime}$ has type $2$, then%
\[
\mathcal{L}(\ell_{p},E_{2};\mathbb{K})=\Pi_{(p;2,1)}(\ell_{p},E_{2};\mathbb{K})=\Pi_{(2p/(2+p);1,1)}%
(\ell_{p},E_{2};\mathbb{K})=\Pi_{(r_{p};r_{p},r_{p})}(\ell_{p},E_{2};\mathbb{K}),
\]
for every $1\leq r_{p}\leq\frac{2p}{3p-2}.$
\end{theorem}

\begin{proof}
We only treat the case $\mathbb{K}=\mathbb{C}.$ The case $\mathbb{K}%
=\mathbb{R}$ follows from a complexification argument (see \cite{Junek-preprint, Thesis} for
details).

Assume first that $p=1$. Let $A\in\mathcal{L}(\ell_{1},E_{2};\mathbb{K})$. Since $E_{2}^{\prime}$ has type
$2,$ it follows from \cite[Theorem 12.10]{Diestel} that $A_{1}\in\Pi_{a.s}(\ell_{1};E_{2}^{\prime})$. So,
from
the previous theorem it follows that%
\[
\hat{A}\colon\ell_{2}^{w}(\ell_{1})\times\ell_{1}^{w}(E_{2})\rightarrow
\ell_{1}%
\]
is bounded. Hence $A\in\Pi_{(1;2,1)}(\ell_{1},E_{2};\mathbb{K})$. On the other
hand, from the inclusion theorem we know that $\mathcal{L}(\ell_{2}%
,E_{2};\mathbb{K})=\Pi_{(2;2,1)}(\ell_{2},E_{2};\mathbb{K})$.

Let now $1\leq p\leq2$ and $A\in\mathcal{L}(\ell_{p},E_{2};\mathbb{K})$. Fix $(y_{j})\in\ell_{1}^{w}(E_{2})$ and consider the linear mappings%
\[
T^{(1)}\colon\ell_{2}^{w}(\ell_{2})\rightarrow\ell_{2}\text{ and }%
T^{(2)}\colon\ell_{2}^{w}(\ell_{1})\rightarrow\ell_{1}\text{ }%
\]
given by
\[
T^{(k)}((x_{j})_{j})=(A(x_{j},y_{j}))_{j}\text{ for }k=1,2.
\]
Clearly $T^{(1)}$ and $T^{(2)}$ are well-defined and continuous. Using that
$\ell_{2}^{w}(\ell_{t})=\mathcal{L}(\ell_{2};\ell_{t})$ for $t=1,2$, \cite[proof of the
Theorem]{Pisier} gives that
\[
\ell_{2}^{w}(\ell_{p})\subset(\ell_{2}^{w}(\ell_{2}),\ell_{2}^{w}(\ell
_{1}))_{\theta}%
\]
for $\frac{\theta}{2}=1-\frac{1}{p}$. So the complex interpolation method implies that%
\begin{align*}
T &  \colon\ell_{2}^{w}(\ell_{p})\rightarrow\ell_{p}\\
T((x_{j})_{j}) &  =(A(x_{j},y_{j}))_{j}%
\end{align*}
is continuous. It follows that $A\in\Pi_{(p;2,1)}(\ell_{p},E_{2};\mathbb{K}%
)$. Since $E_{2}^{\prime}$ has type $2$, it follows from \cite[page 220]{Diestel} that $E_{2}$ has cotype
$2$. Now use Theorem \ref{cotipo} to obtain
$\Pi_{(p;2,1)}(\ell_{p},E_{2};\mathbb{K})=\Pi_{(2p/(2+p);1,1)}(\ell_{p},E_{2};\mathbb{K})$.
Using the inclusion theorem once again one has
$$\Pi_{(\frac{2p}{2+p};1,1)}%
(\ell_{p},E_{2};\mathbb{K})\subset\Pi_{(s_{p};s_{p},s_{p})}(\ell_{p},E_{2};\mathbb{K})$$ for
$2-\frac{2+p}{2p}=\frac{1}{s_{p}}$, which gives us $s_{p}=\frac{2p}{3p-2}$. So, since $1\leq s_{p}\leq2,$
from \cite[Theorem
3]{Junek} it follows that $\Pi_{(r_{p};r_{p},r_{p})}(\ell_{p},E_{2};\mathbb{K}%
)=\mathcal{L}(\ell_{p},E_{2};\mathbb{K})$ whenever $1\leq r_{p}\leq
s_{p}.$
\end{proof}

\begin{corollary}
\label{lp1} If $1\leq p\leq2$ and $1<q\leq2$ then

(i)
$\mathcal{L}(\ell_{p},\ell_{q};\mathbb{K})=\Pi_{(p;2,1)}(\ell_{p},\ell_{q};\mathbb{K}).$

(ii)
$\mathcal{L}(\ell_{1},\ell_{q};\mathbb{K})=\Pi_{(r;r,r)}(\ell_{1},\ell_{q};\mathbb{K})$
for $1\leq r\leq2$.
\end{corollary}

The following result (for $n$-linear mappings) can also be obtained using
results from \cite{Port} and the idea of the proof of Theorem \ref{t2}.

\begin{proposition}
Let $n\geq2$ and $1<p\leq2.$ Then every $n$-linear mapping
$A\in\mathcal{L}(\ell_{1},\overset{n-1}{\ldots},\ell_{1},\ell_{p};\mathbb{K})$ is
$(r_{n};r_{n},\ldots,r_{n})$-summing for every $1\leq r_{n}\leq \frac{2^{n-1}}{2^{n-1}-1}$.
\end{proposition}

\begin{proof}  The case $n=2$ is proved in Corollary \ref{lp1}(ii). From
\cite[Theorem 3 and Remark 2]{Junek} it suffices to prove the result for $r_{n}=\frac{2^{n-1}}{2^{n-1}-1}.$

Case $n=3$ and $\mathbb{K}=\mathbb{C}$: Let
$A\in\mathcal{L}(\ell_{1},\ell _{1},\ell_{p};\mathbb{K}).$ From
Corollary \ref{lp1}(i) we know that
\begin{equation}
\mathcal{L}(\ell_{1},\ell_{p};\mathbb{K})=\Pi_{(1;2,1)}(\ell_{1},\ell_{p};\mathbb{K}). \label{31Dez}%
\end{equation}
From (\ref{31Dez}) and \cite[Corollary 3.2]{Port} we get
\[
\mathcal{L}(\ell_{1},\ell_{1},\ell_{p};\mathbb{K})=\Pi_{(1;2,1,1)}(\ell_{1},\ell_{1}%
,\ell_{p};\mathbb{K})=\Pi_{(1;1,2,1)}(\ell_{1},\ell_{1},\ell_{p};\mathbb{K}).
\]
So,%
\begin{equation}
\widehat{A}\colon\ell_{2}^{u}(\ell_{1})\times\ell_{1}^{u}(\ell_{1})\times
\ell_{1}^{u}(\ell_{p})\rightarrow\ell_{1}~ \label{DD1}%
\end{equation}
is bounded. Combining now Corollary \ref{lp1}(ii) with
\cite[Corollary 3.2]{Port} we get
that%
\begin{equation}
\widehat{A}\colon\ell_{1}^{u}(\ell_{1})\times\ell_{2}^{u}(\ell_{1})\times
\ell_{2}^{u}(\ell_{p})\rightarrow\ell_{2} \label{DD2}%
\end{equation}
is bounded. So, using complex interpolation for (\ref{DD1}) and
(\ref{DD2}) we conclude that%
\[
\widehat{A}\colon\ell_{4/3}^{u}(\ell_{1})\times\ell_{4/3}^{u}(\ell_{1}%
)\times\ell_{4/3}^{u}(\ell_{p})\rightarrow\ell_{4/3}%
\]
is bounded (this use of interpolation is based on results of
\cite{DM}, which are closely related to the classical paper
\cite{Kouba} - further details can be found in \cite{Junek}). \\
\indent Case $n=4$ and $\mathbb{K}=\mathbb{C}$: From the case $n=3$
and \cite[Corollary 3.2]{Port} we know that
\begin{equation*}
\mathcal{L}(\ell_{1},\ell_{1},\ell_{1},\ell_{p};\mathbb{K})=\Pi_{(\frac{4}{3};1,\frac
{4}{3},\frac{4}{3},\frac{4}{3})}(\ell_{1},\ell_{1},\ell_{1},\ell_{p};\mathbb{K}).
\label{31Dez2}%
\end{equation*}
Since $\frac{4}{3} < 2$, Corollary \ref{lp1}(i) gives that
$\mathcal{L}(\ell
_{1},\ell_{p};\mathbb{K})=\Pi_{(1;\frac{4}{3},1)}(\ell_{1},\ell_{p};\mathbb{K})$.
So \cite[Corollary 3.2]{Port} implies
\[
\mathcal{L}(\ell_{1},\ell_{1},\ell_{1},\ell_{p};\mathbb{K})=\Pi_{(1;\frac{4}{3}%
,1,1,1)}(\ell_{1},\ell_{1},\ell_{1},\ell_{p};\mathbb{K}).
\]
Hence%
\begin{align*}
\widehat{A}  &  \colon\ell_{1}^{u}(\ell_{1})\times\ell_{\frac{4}{3}}^{u}%
(\ell_{1})\times\ell_{\frac{4}{3}}^{u}(\ell_{1})\times\ell_{\frac{4}{3}}%
^{u}(\ell_{p})\rightarrow\ell_{\frac{4}{3}}~\mathrm{and}\\
\widehat{A}  &  \colon\ell_{\frac{4}{3}}^{u}(\ell_{1})\times\ell_{1}^{u}%
(\ell_{1})\times\ell_{1}^{u}(\ell_{1})\times\ell_{1}^{u}(\ell_{p}%
)\rightarrow\ell_{1}%
\end{align*}
are bounded. Using complex interpolation once more we conclude that%
\[
\widehat{A}\colon\ell_{\frac{8}{7}}^{u}(\ell_{1})\times\ell_{\frac{8}{7}}%
^{u}(\ell_{1})\times\ell_{\frac{8}{7}}^{u}(\ell_{1})\times\ell_{\frac{8}{7}%
}^{u}(\ell_{p})\rightarrow\ell_{\frac{8}{7}}%
\]
is bounded as well. The cases $n>4$ are similar and the real case
follows by complexification.
\end{proof}

\vspace{2mm}

\noindent[Oscar Blasco] Departamento de An\'alisis Matem\'atico, Universidad
de Valencia, 46.100 Burjasot - Valencia, Spain, e-mail: oscar.blasco@uv.es

\medskip

\noindent[Geraldo Botelho] Faculdade de Matem\'atica, Universidade Federal de
Uberl\^andia, 38.400-902 - Uberl\^andia, Brazil, e-mail: botelho@ufu.br

\medskip

\noindent[Daniel Pellegrino] Departamento de Matem\'atica, Universidade Federal da Para\'iba, 58.051-900 -
Jo\~ao Pessoa, Brazil, e-mail: dmpellegrino@gmail.com

\medskip

\noindent[Pilar Rueda] Departamento de An\'alisis Matem\'atico, Universidad de
Valencia, 46.100 Burjasot - Valencia, Spain, e-mail: pilar.rueda@uv.es

\end{document}